\documentclass[12pt]{article} 
\usepackage{cmap}
\usepackage{ucs}
\usepackage[utf8]{inputenc}
\usepackage[top=2.5cm,bottom=2.5cm,left=2.5cm,right=2.5cm]{geometry}
\usepackage{amsfonts,amsmath,amsthm,amscd,amssymb,latexsym}
\usepackage{mathtools}
\usepackage[english]{babel}
\usepackage{tikz}
\usepackage{subcaption}
\usetikzlibrary{positioning}
\usetikzlibrary{intersections}
\usetikzlibrary{arrows.meta, decorations.markings}
\usetikzlibrary{shapes.geometric, arrows}
\usepackage{float}
\usepackage{csquotes}
\usepackage{enumitem}
\usepackage{comment}
\usepackage{adjustbox}
\usepackage[labelfont=bf, labelsep=period]{caption}
\usepackage[colorinlistoftodos]{todonotes}
\usepackage[output-decimal-marker={,}]{siunitx}
\usepackage{listings}
\usepackage{xcolor}
\usepackage{array}
\usepackage{anyfontsize}
\usepackage{ifthen}
\usepackage{authblk}
\usepackage{enumitem}
\usepackage{relsize}

\mathtoolsset{showonlyrefs=true}

\usepackage[pdftex,colorlinks,linkcolor=blue,citecolor=red,urlcolor=blue,hyperindex,plainpages=false,bookmarksopen,bookmarksnumbered,unicode]{hyperref}

\usepackage{attachfile}

\usepackage[table]{xcolor}

\definecolor{nocellthree}{RGB}{245,220,220}   
\definecolor{nocelltwo}{RGB}{235,180,180}   
\definecolor{nocellone}{RGB}{215,140,140} 
\definecolor{nocellfour}{RGB}{245, 143, 225} 

\newcommand{\No}[1]{%
  \ifcase#1\relax
  \or\cellcolor{nocellone}No(1)%
  \or\cellcolor{nocelltwo}No(2)%
  \or\cellcolor{nocellthree}No(3)%
  \or\cellcolor{nocellfour}No(4)%
  \fi
}

\definecolor {processblue}{cmyk}{0.96,0,0,0}
\definecolor{codegreen}{rgb}{0,0.6,0}
\definecolor{codegray}{rgb}{0.5,0.5,0.5}
\definecolor{codepurple}{rgb}{0.58,0,0.82}
\definecolor{backcolour}{rgb}{0.95,0.95,0.92}

\makeatletter 
\newcount\SOUL@minus
\makeatother  

\lstdefinestyle{mystyle}{
	backgroundcolor=\color{backcolour},   
	commentstyle=\color{codegreen},
	keywordstyle=\color{blue},
	numberstyle=\tiny\color{codegray},
	stringstyle=\color{codepurple},
	basicstyle=\ttfamily\footnotesize,
	breakatwhitespace=false,         
	breaklines=true,                 
	captionpos=b,                    
	keepspaces=true,                 
	numbers=left,                    
	numbersep=5pt,                  
	showspaces=false,                
	showstringspaces=false,
	showtabs=false,                  
	tabsize=2
}

\newtheorem{theorem}{Theorem}[section]
\newtheorem{lemma}[theorem]{Lemma}
\newtheorem{proposition}[theorem]{Proposition}
\newtheorem{corollary}[theorem]{Corollary}

\newtheorem{observation}[theorem]{Observation}

\theoremstyle{definition}
\newtheorem{definition}[theorem]{Definition}

\begin{document}
\large
\title{\vspace{-2cm} Regular $K_3$-regular graphs}
	
\author[1, 2]{Artem Hak \thanks{Corresponding author: artikgak@ukr.net}}
\author[2]{Sergiy Kozerenko}
\author[3]{Denys Lohvynov}
\author[2, 4]{Yurii Yarosh}

\affil[1]{National University of Kyiv-Mohyla Academy, 2 Skovorody Str., 04070 Kyiv, Ukraine}

\affil[2]{Kyiv School of Economics, 3 Mykoly Shpaka Str., 03113 Kyiv, Ukraine}

\affil[3]{National Technical University of Ukraine
``Igor Sikorsky Kyiv Polytechnic Institute'', 37 Beresteiskyi Ave., 03056 Kyiv, Ukraine}

\affil[4]{Kyiv Academic University, 36 Vernadsky Blvd., 03142 Kyiv, Ukraine}

	\date{\vspace{-5ex}}
	
	\maketitle
	
	\begin{abstract}
    We study graphs that are simultaneously regular with respect to the ordinary vertex degree and regular with respect to the triangle degree, that is, the number of triangles containing a given vertex. We call such graphs regular $K_3$-regular. We investigate the (non-)existence of regular $K_3$-regular graphs with prescribed parameters $(r_2,r_3)$, where $r_2$ is the vertex degree and $r_3$ is the triangle degree. General bounds relating vertex and edge triangle degrees are derived, and non-existence results are established for broad ranges of these parameters. Special attention is paid to Tur\'an graphs, for which we establish uniqueness results for certain parameters. The paper concludes with a summary of admissible parameters and several open problems.
	\end{abstract}
	{ \small
	}

\section{Introduction}

A graph is called regular if its vertices have the same degree. A simple application of the pigeonhole principle implies that any graph with at least two vertices has two vertices with the same degree. Thus, there are no ``irregular'' graphs. Nevertheless, in order to approach a sensible definition of irregularity for graphs, Chartrand, Erd\H{o}s, and Oellermann~\cite{Char-Erd-Oell:88} proposed to use the notion of $F$-degree defined in the work~\cite{Char:87}. There, for a fixed graph $F$, the $F$-degree of a vertex $u$ from a graph $G$ is the number of subgraphs of $G$ isomorphic to $F$ to which $u$ belongs. This way, the standard vertex degree is just the $K_{2}$-degree.

In contrast, for other graphs $F\neq K_{2}$, there can exist $F$-irregular graphs (i.e., graphs $G$ with all vertices having pairwise distinct $F$-degrees). 
For the smallest possible $K_{3}$-irregular graph, see the work~\cite{dist-triangle:2024}. In~\cite{Char-Erd-Oell:88}, the question of whether there exists a regular $K_{3}$-irregular graph was posed. This problem remained open for several decades until, in 2024, Stevanovi\'c et al.~\cite{reg-triangle:24} found a $10$-regular $K_{3}$-irregular graph. More recently, the preprint~\cite{Hak:25} gives an example of a $9$-regular $K_{3}$-irregular graph obtained using genetic algorithm techniques. It is also shown that there are no such graphs for regularities at most $7$, with the case of regularity $8$ constituting the most interesting yet open case (again, see~\cite{Hak:25} for the details). For other aspects of irregularity in graphs, we refer to~\cite{Chartrand2016}.

In this paper, we consider regular $K_{3}$-regular graphs. Our motivation is to further develop the understanding of different regularities a graph can exhibit. The main focus is on the (non-)existence of regular $K_{3}$-regular graphs for the given parameters of regularity. 

A related framework was introduced in~\cite{Jajcay2025}, where the authors study 
vertex-girth-regular graphs, namely $\operatorname{vgr}(n;k;g;\lambda)$-graph is 
a $k$-regular graph of order $n$ and girth $g$ in which every vertex belongs to exactly $\lambda$ 
cycles of length $g$. Thus, the graph is regular not only with respect to vertex degrees but also 
with respect to the number of shortest cycles containing each vertex. Several non-existence results were proved for odd $g \ge 7$ and even $g \ge 4$ (see~\cite[Theorems 17, 18]{Jajcay2025}). For the cases $g=3$ and $g=5$, the authors formulate open questions asking whether the non-existence results from the mentioned theorems hold for these values of girth as well. The present paper focuses on regular $K_3$-regular graphs, which corresponds precisely to the case $g=3$.

This class of graphs has also already appeared in the literature under a different name. In~\cite{Caro2024}, regular $K_3$-regular graphs with parameters $(r_2,r_3)$ are called $(r_2,r_3)$-constant graphs. Non-existence result for $r_3 = {r_2 \choose 2} - k$ with $k\in\mathbb{N}$ and $3k\le r_2$ was proven (see~\cite[Theorem 14]{Caro2024}). In~\cite{CaroMifsud2025}, the existence of planar $(r_2, r_3)$-constant graphs and circulant $(r_2, r_3)$-constant graphs is studied. In addition, the authors completely answer the question of the existence of $(r_2, r_3)$-constant graphs for $r_2\le 6$, $r_3\le {r_2 \choose 2}$.

Therefore, some basic constraints and structural properties of such graphs are already known; one of them is the following.

\begin{observation}
For every $r_3 \in \mathbb{N} \cup \{0\}$, there exists an integer $r_2^*\in \mathbb{N} \cup \{0\}$ such that
\begin{enumerate}[label={$(\arabic*)$}]
\item for all $r_2 \geq r_2^*$, there exists a graph with parameters $(r_2, r_3)$;
\item for all $r_2 < r_2^*$, no graph with parameters $(r_2, r_3)$ exists.
\end{enumerate}
\end{observation}

This observation is a simple corollary of a trivial fact that $K_2$ and $K_3$ are regular $K_3$-regular, and the well-known fact that the degree and $K_3$-degree of a vertex are both additive under the Cartesian product. As such, we see that $r_2^* \leq O(\sqrt{r_3})$ (consider $K_n \square K_3 \square \ldots \square K_3$). Also, it is worth mentioning that authors in \cite{Caro2024} improved this bound to $r_2^* \leq f^{-1}(r_3)$, where $f(x) = \frac{x^2}{2}-5x^{\frac{3}{2}}$ (see~\cite[Theorem 14]{Caro2024}).

The paper is organised as follows. In Section~\ref{sect-defs-prelim}, we give all the necessary definitions that will be used in the paper. Section~\ref{sect-bounds} is devoted to establishing main technical bounds for vertex and edge $K_{3}$-degrees in regular $K_{3}$-regular graphs. The important results on the non-existence of such graphs are provided by Proposition~\ref{Cr:noDown} and Theorem~\ref{Pr:noDown}.  Section~\ref{sect-Turan} contains results on a well-known family of Tur\'an graphs, which are proved to be an important class of regular $K_{3}$-regular graphs. In particular, we present a characterisation of Tur\'an graphs for certain parameters in Theorem~\ref{cor:5.3} and in Theorem~\ref{th:K333} (which also demonstrates non-existence). We conclude the paper with Section~\ref{sect-4}, presenting several open questions and research directions on regular $K_{3}$-regular graphs. Finally, Appendix~\ref{appendix} contains some examples of regular $K_{3}$-regular graphs for small regularities, and a table which summarises our results on the existence of these graphs.

\section{Main definitions}\label{sect-defs-prelim}

A graph is an ordered pair $G = (V, E)$, where $V = V(G)$ is the set of its \textit{vertices} and $E = E(G) \subset \binom{V}{2}$ is the set of its \textit{edges}. All the graphs considered in this paper are undirected, simple and finite. Also, for a pair of vertices $u, v \in V(G)$, the edge $\{u, v\}$ will be denoted as $uv$. By $\overline{G}$, we denote the \textit{complement} of a graph $G$. The complete graph with $n$ vertices is denoted by~$K_n$.

The \textit{open neighbourhood} of a vertex $v$ in a graph $G$ is the set of all its adjacent vertices: $N(v) = \{u \in V(G): uv \in E(G)\}$. The \textit{closed neighbourhood} of $v$ is the set $N[v] = N(v) \cup \{v\}$. The \textit{degree} of $v$ in $G$ is the number $\deg(v) = |N(v)|$. A vertex $v \in V(G)$ is an \textit{isolated vertex} provided $\deg(v) = 0$. The graph $G$ is said to be \textit{regular} if all its vertices have the same degree, which is called the \textit{regularity} of $G$. 

Two graphs $G$ and $H$ are called \textit{isomorphic} if there is an \textit{isomorphism} between them, that is, a bijection $f\colon V(G) \rightarrow V(H)$ such that $uv \in E(G)$ if and only if $f(u)f(v) \in E(G)$.

Given vertices $v_1, \ldots, v_k \in V(G)$, by $G[v_1, \ldots, v_k]$ we denote the subgraph induced by the union of their closed neighbourhoods $\bigcup_{i=1}^k N[v_i]$, and by $G(v_1, \ldots, v_k)$ we denote the subgraph induced by the union of their open neighbourhoods $\bigcup_{i=1}^k N(v_i)$.

The \textit{Cartesian product} of two graphs $G_1$ and $G_2$ is the graph $G_1 \square G_2$ with $V(G_1 \square G_2) = V(G_1)\times V(G_2)$, and two vertices $(u, v)$ and $(u', v')$ being adjacent if and only if either $u=u'$ and $vv' \in E(G_2)$, or $v=v'$ and $uu' \in E(G_1)$.

\begin{definition}[\cite{Char-Erd-Oell:88, Char:87}]
For a given graph $F$, the \textit{$F$-degree} of a vertex $v$ in~$G$ is the number $F \deg(v)$ of subgraphs of $G$, isomorphic to $F$, to which $v$ belongs.
\end{definition}

Note that the ordinary degree of a vertex is exactly its $K_2$-degree. Also, it is readily seen that $K_3\deg(v)=|{N(v)\choose 2}\cap E(G)|$.

\begin{definition}
For a given graph $F$, the \textit{$F$-degree} of an edge $e\in E(G)$ is the number $F \deg(e)$ of subgraphs of $G$, isomorphic to $F$, to which $e$ belongs.
\end{definition}

Note that $K_3\deg(uv)  = |N(u)\cap N(v)|$ for any edge $uv \in E(G)$. 

\begin{definition}
A graph is called $K_3$-regular if all its vertices have the same $K_3$-degree.  
\end{definition}

Throughout the paper, we focus on graphs that are both regular and $K_3$\nobreakdash-regular. If $G$ is such a graph, we denote by $r_2$ the common $K_2$-degree (i.e., the ordinary degree) of all its vertices, and by $r_3$ the common $K_3$-degree of all its vertices. We say that $G$ \textit{has parameters} $(r_2, r_3)$ if $G$ is regular with degree $r_2$ and $K_3$-regular with $K_3$-degree $r_3$.

\begin{definition}
Tur\'an graphs are complete multipartite graphs built by partitioning a set of~$n$ vertices into $r$ subsets as equally as possible and connecting every pair of vertices that lie in different parts.    
\end{definition}

 In this paper, we restrict ourselves to the case where $r \mid n$, so that all parts have equal size (also known as complete multipartite graphs). These graphs were introduced in Tur\'an's classical 1941 paper~\cite{Turan1941}. We shall denote these graphs by $\operatorname{Turan}(n, r)$.

Assume that $r \mid n$, then Tur\'an graphs are regular and, moreover, $K_3$-regular. For $r\geq 3$, we can explicitly calculate the regularity parameters
\[r_2=\frac{n}{r}(r-1),
\qquad
r_3={ r-1 \choose 2}\cdot \left(\frac{n}{r}\right)^2.\]

\section{Establishing bounds}\label{sect-bounds}

In this section, we establish fundamental bounds on the parameters of $K_3$-regular graphs. We derive several direct estimates and develop key technical lemmas relating edge $K_3$-degrees to vertex degrees. These results form the main tools for the non-existence theorems proved in the subsequent sections.

The following lemma gives a trivial but useful upper bound on the parameter $r_3$ in terms of $r_2$. 

\begin{lemma}\label{K3_upper_bound}
For every graph with parameters $(r_2, r_3)$, the following upper bound holds:
\begin{equation}~\label{K3_upper_bound_formula}
   r_3 \leq \binom{r_2}{2}.
\end{equation}
\end{lemma}

\subsection{Edge degree statements}

The following Lemmas 
\ref{K3_deg_eq},
\ref{K3_edge_upper_bound} and \ref{K3_edge_lower_bound} are considered instrumental since what they provide is several relations between the edge $K_3$-degree and vertex degrees without giving any tangible constraints on the parameters right away. Nevertheless, they will play a crucial role in subsequent existence and non-existence arguments.

The following lemma follows directly from the handshaking lemma, and it will be used several times throughout the article.

\begin{lemma}\label{K3_deg_eq}
For any graph $G$ and any vertex $v$, we have
    \[K_3 \deg(v) = \frac12 \sum_{u \in N(v)} K_3 \deg(uv). \]
\end{lemma}

\begin{proof}
    Consider the induced subgraph $H = G(v)$. Each triangle of $G$ containing $v$ corresponds to an edge of $H$, therefore $K_3\deg_{G}(v) = |E(H)|$. Moreover, for every $u \in V(H)$, we have 
    \[K_3 \deg_{G}(uv)=\deg_{H}(u).\] 
    Applying the handshaking lemma to $H$ yields 
    \[\frac12 \sum_{u \in V(H)} K_3 \deg_G(uv) = \frac12\sum_{u \in V(H)} \deg_{H}(u) = |E(H)| = K_3 \deg_G(v),\] which completes the proof.
\end{proof}

The next lemma gives an upper bound for the $K_{3}$-degrees of edges in regular $K_{3}$-regular graphs.

\begin{lemma}\label{K3_edge_upper_bound}
Let $G$ be a graph with parameters $(r_2, r_3)$. Then, for any edge $e\in E(G)$, the following upper bound holds:
\begin{equation}~\label{K3_upper_bound_edge_formula}
    K_3 \deg(e) \leq \min\{r_3, r_2-1\}.
\end{equation}
\end{lemma}
\begin{proof}
Fix an edge $e=uw \in E(G)$.
By definition,
$K_3 \deg(e) = |N(u)\cap N(w)|$. Since, $N(u)\cap N(w) \subseteq N(u)\cap (N(w)\setminus \{u\})$, we obtain \[K_3 \deg(e) \leq |N(w)\setminus \{u\}| \leq r_2-1.\]

On the other hand, each vertex $v \in N(u)\cap N(w)$ determines a triangle $\{u,w,v\}$ containing $u$. Hence, $K_3 \deg(e) \leq K_3\deg(u) = r_3$, and the claim follows.
\end{proof}

The following lemma complements the previous result by relating edge $K_{3}$-degrees to the regularity parameters.

\begin{lemma}\label{K3_edge_lower_bound}
    Let $G$ be a graph with parameters $(r_2, r_3)$ and let $e \in E(G)$ be an edge with $K_3 \deg(e) = k$. Then the following inequality holds:
    \[2r_3 \leq (r_2 - k - 1)(r_2-2) + k(k+1).\]
\end{lemma}
\begin{proof}
    Let $G$ be a graph with parameters $(r_2, r_3)$ and let $uv \in E(G)$. Set \[K_3 \deg(uv) = |N(u)\cap N(v)|=k.\] For a vertex $x \in V(G)$, denote by $\triangle(x)$ the set of all triangles containing $x$. Our first step is to estimate the size of the set $\triangle(u)\cup \triangle(v)$.

    Define 
    \[C = N(u)\cap N(v) = \{w_1, \ldots, w_k\},\quad I_1 = N(u)\setminus N[v],\quad I_2 = N(v)\setminus N[u],\] 
    put $i = |I_1|=|I_2| = r_2-k-1$. Let $G'$ be a subgraph induced by $C$, and for each $w_j \in C$ set $d_j = \deg_{G'}(w_j)$. Note that $0\leq d_j\leq k-1$.

    Now, having taken care of the prerequisites, we shall proceed by calculating $|\triangle(u)\cup \triangle(v)|$. On the one hand, \[|\triangle(u)\cup \triangle(v)| = |\triangle(u)| + |\triangle(v)| - |\triangle(u) \cap\triangle(v)| = r_3+r_3 - |C| = 2r_3-k.\] 
    
    On the other hand, let $\triangle ABC$ denote the set of triangles having one vertex in each of $A$, $B$, and $C$ (where single vertices are identified with singleton sets). We decompose the set $\triangle(u)\cup \triangle(v)$ into the following disjoint subsets:
    \begin{align*}
    \triangle(u)\cup \triangle(v) = (\triangle I_1 I_1 u \sqcup \triangle I_2 I_2 v) \sqcup \triangle C uv &\sqcup (\triangle CCu \cup \triangle CCv) \\
    &\sqcup (\triangle CI_1u \cup \triangle CI_2v).
    \end{align*}
    
    We estimate the cardinalities of these sets separately:
    \begin{itemize}
    \item Clearly, \[|\triangle I_1 I_1 u \sqcup \triangle I_2 I_2 v| \leq 2 {i \choose 2}.\] 
    
    \item The number of triangles containing $u$, $v$ and one vertex of $C$ is exactly 
    \[|\triangle C u v| = |C|=k.\]
    
    \item Each edge of $G'$ contributes to exactly two triangles, one containing $u$ and one containing $v$, therefore
    \[|\triangle CCu \cup \triangle CCv| = 2|E(G')|=\sum_{j=1}^k d_j.\]
    
    \item Finally, consider the term $|\triangle C I_1 u \cup \triangle C I_2 v|$. Each vertex $w_j$ is adjacent to $d_j$ vertices from $C$ and to $u$ and $v$. Hence, $w_j$ cannot contribute more than $r_2-d_j-2$ triangles of the aforementioned type. 
    
    Thus, \[|\triangle CI_1u \cup \triangle CI_2v|\leq \sum_{j=1}^k (r_2-d_j-2).\]
    \end{itemize}

    Summing it all up, we obtain 
    \begin{align*}
        2r_3-k = |\triangle(u)\cup \triangle(v)| & \leq 2 {i\choose 2} +k + \sum_{j=1}^k d_j + \sum_{j=1}^k (r_2-d_j-2) \\ &= 2 {i\choose 2} + (r_2-1)k.
    \end{align*}
    Substituting $i=r_2-k-1$ and rearranging yields the stated inequality. 
\end{proof}

\subsection{Non-trivial bounds on the parameters}

The following two statements provide some non-trivial restrictions on the parameters $(r_2, r_3)$ for which there exists a regular $K_{3}$-regular graph. 
It is worth noting that \cite[Proposition~14]{CaroMifsud2025}  states the non-existence of graphs with parameters $(5, 7), (5, 8)$ and $(6, 11)$ by looking specifically at their inner structure. Our next results vastly generalise this fact.
The next proposition is proved by virtue of combining the three lemmas from the previous subsection. 

\begin{proposition}\label{Cr:noDown}
    There are no graphs with parameters $(r_2, r_3)$ such that \[r_3 = {r_2 - 1 \choose 2} + c\] for $r_2 > 4$ and $0 < c < \frac{r_2 - 2}{2}$.
\end{proposition}
\begin{proof}
Suppose that $G$ is a regular $K_3$-regular graph with $r_3 ={r_2 - 1 \choose 2}+c$. Fix an arbitrary edge $e \in E(G)$ and put $k = K_3 \deg(e)$. 
By virtue of Lemma~\ref{K3_edge_upper_bound}, we have $k\leq r_2-1$. Moreover, Lemma~\ref{K3_edge_lower_bound} yields \[2\left({r_2 - 1 \choose 2}+c\right) = 2r_3 \leq (r_2 - 1 -k)(r_2 - 2)+k(k+1).\]
Solving the quadratic inequality in $k$, we obtain \[k\geq \frac{r_2 - 3}{2}+\sqrt{2c+\left(\frac{r_2 - 3}{2}\right)^2} > r_2 - 3\] or \[k\leq \frac{r_2 - 3}{2} - \sqrt{2c + \left(\frac{r_2 - 3}{2} \right)^2} < 0.\]
Since $k \geq 0$, the latter case is impossible and, therefore, $k \in \{r_2 - 2, r_2 - 1\}$. 

Fix an arbitrary vertex $v \in V(G)$. Among the $r_2$ edges incident to $v$, let $a$ and $b$ denote the numbers of edges $vu$ of $K_3$-degree $r_2 - 2$ and $r_2 - 1$, respectively.
Then $a+b = r_2$. 
By Lemma \ref{K3_deg_eq}, we obtain
 \begin{align*}
 2 K_3 \deg(v) &= \sum_{u \in N(v)} K_3 \deg(uv),\\
 2r_3 = 2\left({r_2 - 1 \choose 2}+c\right)&=(r_2 - 2) a + (r_2 - 1) b,\\ a+b&=r_2, \quad a, b \geq 0. 
\end{align*}

Using $c < \frac{r_2 -2}{2}$, we have
\begin{align*}
(r_2 - 2)a + (r_2 - 1)b = 2 \left({r_2 - 1 \choose 2} + c \right) &< r_2^2 - 2r_2,\\
r_2(a+b) - (a + b) - a &< r_2^2 - 2r_2,\\
r_2 &< a.
\end{align*}
This leads to a contradiction. Hence, there are no such graphs.
\end{proof}

The following theorem establishes the non-existence result for a range of regularity parameters. In particular, it excludes such pairs of parameters $(r_2,r_3)$ as $(4, 5)$ and $(5, 9)$, which appear in Table~\ref{tab:admissible} in Appendix~\ref{appendix} (see Corollaries~\ref{cor:5.1} and~\ref{cor:5.2}). Our theorem improves upon \cite[Theorem 2]{SheffieldXi:25}; in particular, we exclude one more parameter for every odd $r_2$ by using a different proof technique. Moreover, this result affirmatively answers the question posed in \cite{Jajcay2025} regarding the existence of $\operatorname{vgr}(n, r_2, 3, r_3 )$-graphs (see Theorems 17, 18 and the discussion afterwards therein).

\begin{theorem}\label{Pr:noDown}
    There exists no graph with parameters $(r_2, r_3)$ such that $r_3 = {r_2 \choose 2} - c$, provided $r_2 \geq 3$ and $1\leq c\leq\frac{r_2-1}{2}$.
\end{theorem}
\begin{proof}
    Suppose, for the sake of contradiction, that such a graph $G$ exists. 
    Let $v \in V(G)$ be an arbitrary vertex and let $G[v]$ be the subgraph induced by the closed neighbourhood of $v$. Note that the $K_3$-degree of $v$ equals the number of edges in its open neighbourhood, namely, \[|E(G(v))|= K_3 \deg(v) = {r_2 \choose 2} - c.\] 
    Therefore, the minimum degree in $G[v]$ is at least $r_2 - c$. Fix a vertex $u \in N[v]$ of such degree in $G[v]$.  Hence, $\deg_{G[v]}(u) = r_2-d$, where $1 \leq d \leq c$. 
    
    Since $\deg_{G[v]}(u) < r_2$, there exists a vertex $w \in V(G) \setminus N[v]$ such that $wu \in E(G)$. 
    For $k \in \{r_2 - d, \ldots, r_2 - 1 \}$, define 
    \[B_k = \{v' \in N(v): \deg_{G[v]}(v') = k\},
    \qquad
    B_k' = \{v'\in B_k: v'w\in E\}.\]
    Note that $u \in B_{r_2-d}'$ by construction (see Figure~\ref{fig:for:thm:1}).
    Also, denote their unions as
    \[
    B = \bigcup_{k=r_2-d}^{r_2-1} B_k
    \qquad \text{and} \qquad
    B'= \bigcup_{k=r_2-d}^{r_2-1} B_k'.
    \]
    The handshaking lemma applied to the complement of $G[v]$ yields 
    \begin{equation}\label{Bto2c}
        \sum_{j=r_2-d}^{r_2-1} (r_2-j)|B_j|=2c.
    \end{equation}

    In order to reach the contradiction, we will find an upper bound for the $K_3$-degree of $w$.
    For $b_k = |B_k'|$, we have the following constraints: 
    \[b_k \geq 0,~ b_{r_2-d} \geq 1,~ b_{r_2-1}+2b_{r_2-2}+\ldots+db_{r_2-d} \leq 2c.\]
    Denote $I = N(w) \setminus N[v]$, and observe the following:
    \begin{equation}\label{eqr2i}
    |I|+b_{r_2-1}+\ldots+b_{r_2-d} = \deg_G(w) = r_2.
    \end{equation}

    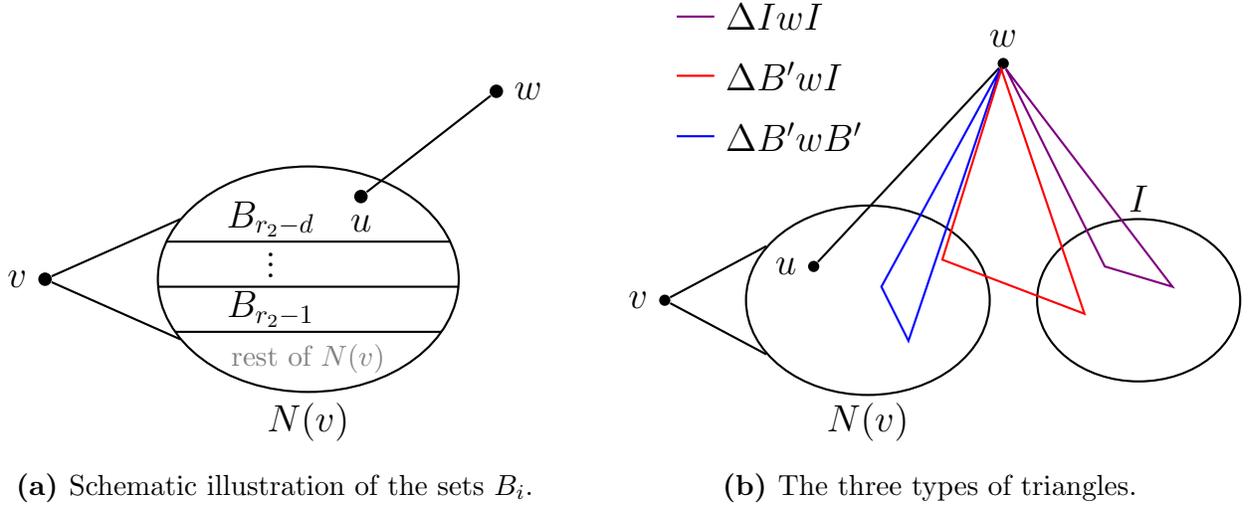
\begin{figure}[htbp]
    	\centering
    	\begin{subfigure}{0.45\textwidth}
    		\centering
    		\begin{tikzpicture}[
    			vertex/.style={circle, fill=black, inner sep=1.8pt},
    			important/.style={circle, fill=black, inner sep=1.8pt},
    			>=Stealth,        
    			font=\small
    			]
    			
    			\node[vertex, label=left:{\large $v$}] (v) at (0,0) {};
    			\node[vertex, label=right:{\large $w$}] (w) at (6, 2.5) {};
    			
    			\draw[thick] (3.5,0) ellipse (2cm and 1.5cm);
    			\node at (3.5, -1.9) {\large $N(v)$};
    			
    			\draw[thick] (v) -- (1.8, 0.8);
    			\draw[thick] (v) -- (1.8, -0.8);
    			
    			\begin{scope}
    				\clip (3.5,0) ellipse (2cm and 1.5cm);
    				
    				\draw[thick] (1, 0.5) -- (6, 0.5);
    				\draw[thick] (1, -0.1) -- (6, -0.1);
    				\draw[thick] (1, -0.7) -- (6, -0.7);
    				
    				\node at (3, 0.8) {\large $B_{r_2-d}$};
    				\node at (3, 0.3) {\large $\vdots$};
    				\node at (3, -0.4) {\large $B_{r_2-1}$};
    				
    				\node[font=\small, text=gray] at (3.5, -1.05) {rest of $N(v)$};
    			\end{scope}
    			
    			\node[important, label=below:{\large $u$}] (u) at (4.2, 1.1) {};
    			
    			\draw[thick] (u) -- (w);
    			
    		\end{tikzpicture}
    		\caption{Schematic illustration of the sets $B_i$.}
    		\label{fig:for:thm:1}
    	\end{subfigure}
    	\hfill
    	\begin{subfigure}{0.5\textwidth}
    		\begin{tikzpicture}[
    			vertex/.style={circle, fill=black, inner sep=1.5pt},
    			>=Stealth,
    			font=\small,
    			scale=0.90
    			]
    			
    			\coordinate (v_pos) at (0,0);
    			\coordinate (Nv_center) at (3,0);
    			\coordinate (I_center) at (7,0);
    			\coordinate (w_pos) at (5, 3.5);
    			
    			\draw[thick] (Nv_center) ellipse (1.8cm and 1.4cm);
    			\node at (3, -1.8) {\large $N(v)$};
    			
    			\draw[thick] (I_center) ellipse (1.5cm and 1.2cm);
    			\node at (7, 1.5) {\large $I$};
    			
    			\node[vertex, label=left:{\large $v$}] (v) at (v_pos) {};
    			
    			\draw[thick] (v) -- (1.5, 0.8);
    			\draw[thick] (v) -- (1.5, -0.8);
    			
    			\node[vertex, label=left:{\large $u$}] (u) at (2.2, 0.5) {};
    			
    			\node[vertex, label=above:{\large $w$}] (w) at (w_pos) {};
    			
    			\draw[thick] (u) -- (w);
    			
    			
    			\coordinate (b1) at (3.2, 0.2);
    			\coordinate (b2) at (3.6, -0.6);
    			\draw[blue, thick] (w) -- (b1) -- (b2) -- cycle;
    			
    			\coordinate (i1) at (6.5, 0.5);
    			\coordinate (i2) at (7.5, 0.2);
    			\draw[violet, thick] (w) -- (i1) -- (i2) -- cycle;
    			
    			\coordinate (br) at (4.1, 0.6); 
    			\coordinate (ir) at (6.2, -0.2); 
    			\draw[red, thick] (w) -- (br) -- (ir) -- cycle;
    			
    			\begin{scope}[shift={(0, 3.3)}]
    				\matrix[row sep=1.5mm, anchor=west] {
    					\draw[violet, thick] (0,0) -- (0.5,0); & \node[right] {\large $\Delta I w I$}; \\
    					\draw[red, thick] (0,0) -- (0.5,0); & \node[right] {\large $\Delta B' w I$}; \\
    					\draw[blue, thick] (0,0) -- (0.5,0); & \node[right] {\large $\Delta B' w B'$}; \\
    				};
    			\end{scope}
    			
    		\end{tikzpicture}
    		\caption{The three types of triangles.}
    		\label{fig:for:thm:2}
    	\end{subfigure}
    	\caption{Constructions in Theorem~\ref{Pr:noDown}.}
    \end{figure}

    Now we split the triangles containing $w$ into three types (see Figure~\ref{fig:for:thm:2}) and count the size of each type.
    
    \medskip
    \noindent
    \textbf{Type 1:} $\triangle IwI$.
    These are triangles with two vertices in $I$ and one vertex~$w$. Clearly, 
    \begin{equation}\label{trIwI}
    |\triangle IwI| \leq {|I|\choose2}.
    \end{equation}

    \medskip
    \noindent
    \textbf{Type 2:} $\triangle B'wI$.
    These are triangles with one vertex in $I$, one in $B'$ and the vertex $w$. If $x \in B_{r_2-i}'$, then $x$ has exactly $i$ non-neighbours in $G[v]$. Also, $x$ is adjacent to $w$. Hence, $x$ can be adjacent to at most $i-1$ vertices of $I$. Therefore, 
    \begin{equation}\label{trBwI}
    |\triangle B'wI| \leq \sum_{i=1}^{d} (i-1)b_{r_2-i} = b_{r_2-2} + 2b_{r_2-3} + \ldots + (d-1) b_{r_2-d}.
    \end{equation}

    \medskip
    \noindent
    \textbf{Type 3:} $\triangle B'wB'$. These are triangles with two vertices in $B'$ and the vertex~$w$.
    We have 
    \[ \left| \triangle B'wB' \right| = \left| {B'\choose 2}\cap E(G) \right| = \left|{B'\choose 2} \right| - \left|{B'\choose 2}\setminus E(G) \right|.\]
    We shall make use of the following set equality: 
    \[{Y \choose 2} = {X \choose 2}\sqcup {Y\setminus X \choose 2}\sqcup (X\hat{\times}Y),\] where $X \subseteq Y$ and $X\hat{\times}Y=\{\{x, y\}: x\in X,\, y\in Y\setminus X\}$. 
    Apply this set equality to $B'\subseteq B$ and express ${B' \choose 2}$ from it
    \[{B' \choose 2} = {B \choose 2}\setminus\left({B\setminus B' \choose 2}\sqcup (B'\hat{\times}B)\right).\]
    As a result, by virtue of straightforward set-theoretic manipulations, we have that
    \[{B' \choose 2} \setminus E(G) = \left( {B\choose 2} \setminus E(G) \right) \setminus \left( \left( {B\setminus B' \choose 2} \sqcup (B' \hat{\times} B ) \right) \setminus E(G) \right).\]
    
    Thus, using~\eqref{eqr2i} we obtain
    \begin{align}\label{trB'wB'}
    |\triangle B'wB'| &= {b_{r_2-1}+\ldots+b_{r_2-d} \choose 2} - \nonumber \\  &\qquad\quad - \left|\left({B\choose 2}\setminus E(G) \right)\setminus \left(\left({B\setminus B' \choose 2} \sqcup (B'\hat{\times} B)\right)\setminus E(G) \right)\right| \nonumber \\
    &= {r_2-|I|\choose 2} - \left|{B \choose 2}\setminus E(G) \right| + \left|\left({B\setminus B' \choose 2}\!\sqcup (B'\hat{\times}B)\right)\setminus E(G) \right| \nonumber \\
    &= {r_2-|I|\choose 2} - c + \left|\left({B\setminus B' \choose 2}\sqcup (B'\hat{\times}B)\right)\setminus E(G) \right|. 
    \end{align}

    We now estimate the last term. First of all, it is readily seen that
    $${B\setminus B' \choose 2} \sqcup (B' \hat{\times}B) = \bigcup_{x \in B \setminus B'} \left\{\{x, y\}: y \in B \setminus\{x\} \right\}.$$
    Recalling definitions of $B$ and $B'$, we obtain that $B\setminus B' = \bigcup_{i=r_2-d}^{r_2-1}(B_i\setminus B'_i)$. Thus, 
    $${B\setminus B' \choose 2} \sqcup (B'\hat{\times}B) = \bigcup_{i=r_2-d}^{r_2-1} ~ \bigcup_{x \in B_i\setminus B'_i}\left\{\{x, y\}: y \in B\setminus\{x\}\right\}.$$
    Hence, we obtain that
    \[
      \left|\left({B\setminus B' \choose 2} \!\sqcup\! (B'\hat{\times}B)\right)\!\setminus\! E(G)\right| \leq \sum_{i=r_2-d}^{r_2-1} \,\sum_{x \in B_i \setminus B_i'}\! |\{\{x, y\}: y \in B \setminus \{x\}\} \setminus E(G)|.
    \]

    Also, for $x \in B_i \setminus B'_i$, we have $\deg_{N[v]}(x) = i$. Hence,
    \begin{align*}    
     \left|\left\{\{x, y\}: y \in B\setminus\{x\}\right\}\setminus E(G)\right| &= \left|\left\{y\in B\setminus\{x\}: xy\notin E(G)\right\}\right| \\ &\leq
     \left|\left\{y\in N[v]\setminus\{x\}: xy\notin E(G)\right\}\right| \\ &= \deg_{\overline{N[v]}}(x) = r_2 - \deg_{N[v]}(x) = r_2 - i.
    \end{align*}
    
    Thus, we finish the calculations using~\eqref{Bto2c} and~\eqref{trBwI}
     \begin{align}\label{upperBound2c-b}
      \left|\left({B\setminus B' \choose 2}\sqcup (B'\hat{\times}B)\right)\setminus E(G)\right| 
      &\leq \sum_{i=r_2-d}^{r_2-1} (r_2-i) |B_i\setminus B_i'| \nonumber\\ &=
      \sum_{i=r_2-d}^{r_2-1} (r_2 - i)|B_i| - \sum_{i=r_2-d}^{r_2-1}(r_2 - i)|B_i'|
      \nonumber\\ &= 2c - (b_{r_2-1}+2b_{r_2-2}+\ldots+d b_{r_2-d}).
    \end{align}
    
    Therefore, substituting~\eqref{upperBound2c-b} into \eqref{trB'wB'}, we get
    \begin{equation}\label{trBwB}
    |\triangle B'wB'| \leq {r_2-|I| \choose 2} + c - (b_{r_2-1}+2b_{r_2-2}+\ldots+d b_{r_2-d}).
    \end{equation}

    Thus, we have estimated all three triangle types. Summing up inequalities~\eqref{trIwI},~\eqref{trBwI} and~\eqref{trBwB}, we get an upper bound
    \begin{align*}
    K_3 \deg(w) &=|\triangle IwI| + |\triangle B'wI | + |\triangle B'wB' | \\ 
    &\leq 
    \begin{aligned}[t]
    {|I| \choose 2} 
        &+ b_{r_2-2} + 2b_{r_2-3} + \ldots + (d-1) b_{r_2-d} \\ 
    &+ {r_2-|I| \choose 2} + c - (b_{r_2-1}+2b_{r_2-2}+\ldots+d b_{r_2-d}).
    \end{aligned}
    \end{align*}

    Simplifying, we get
    \begin{align*}
    K_3 \deg(w) &\leq {|I| \choose 2} 
     + {r_2-|I| \choose 2} + c - \underbrace{(b_{r_2-1} + b_{r_2-2} + \ldots + b_{r_2-d}}_{\text{this equals } r_2 - |I| \text{ by \eqref{eqr2i} } }) \\
     &={|I| \choose 2} + {r_2-|I| \choose 2} + c-r_{2}+|I|.
    \end{align*}

    Define the function
    \[F(x) = {x \choose 2} + {r_2 - x \choose 2} + c -  r_2 + x = x^2 - x(r_2 - 1) + c + \frac{ r_2^2 - 3r_{2} }{2}.\]

   Now, let us derive stricter constraints on $x:=|I|$. Starting with an upper bound \[x \leq |I| + b_{r_2-1}+\ldots+b_{r_2-d}-1 = r_2-1.\]
    And in a similar manner, we derive a lower bound 
    \begin{align}
      x &\geq |I| - (b_{r_2-2} + 2b_{r_2-3} + \ldots + (d-1)b_{r_2-d}) \\ &= r_2 - (b_{r_2-1}+2b_{r_2-2} + \ldots + db_{r_2-d})\\ &\geq r_2 - 2c.  
    \end{align}

    Hence, we see that
    \[ \max K_3\deg(w)  \leq \max \big\{ F(x): r_2 - 2c \leq x \leq r_2-1\big\}.  \]

    We now determine the maximum of $F$ on this interval. First of all, $F$ is a quadratic function with a positive leading coefficient. Therefore, the maximum is attained at one of the endpoints, namely either $x=r_2-1$ or $x=r_2-2c$
    \[F(r_2-1) - F(r_2-2c) = (r_2-2c)(2c-1)\geq1.\] 
    Thus, the maximum value is $F(r_2-1)={r_2-1 \choose 2} + c-1$. For \(1\le c\le \frac{r_2-1}{2}\), we have
\[
F(r_2-1) < \binom{r_2}{2}-c,
\]
which yields the contradiction
\[K_3 \deg(w) \leq F(r_2-1) < {r_2 \choose 2} - c = r_3.\qedhere\]
\end{proof}

We can summarise the findings of this subsection in the following two corollaries.

\begin{corollary}\label{cor:5.1}
     Let $r_2 \ge 4$ be an even integer. Then there are no graphs with parameters $(r_2,r_3)$, where $4{0.5 r_2\choose 2}<r_3<{r_2\choose 2}$ or ${r_2-1\choose 2}<r_3<4{0.5r_2\choose 2}$.  
    Moreover, one example of a graph that corresponds to $r_3 = {r_2\choose 2}$ is $K_{r_2+1}$, an example that corresponds to $r_3 = 4{0.5r_2\choose 2}$ is $\operatorname{Turan}(r_2+2, 0.5r_2+1)$, and an example that corresponds to $r_3 = {r_2-1\choose 2}$ is $K_{r_2}\square K_2$.      
\end{corollary}
\begin{proof}
    First of all, it is easy to verify that $K_{r_2+1}$, $\operatorname{Turan}(r_2+2, 0.5r_2+1)$ and $K_{r_2}\square K_2$ are the aforementioned examples.

    Now, assume the pair $(r_2, r_3)$ satisfies the first inequality. In this case, we use Theorem~\ref{Pr:noDown}. Indeed, we can represent $r_3$ that satisfies the inequality as ${r_2 \choose 2}-c$, where $c\in \{1, \ldots, {r_2\choose 2}-4{0.5r_2\choose 2}-1\}$. Let us simplify the value ${r_2\choose 2}-4{0.5r_2\choose 2}-1 = \frac{r_2-2}{2}$. So, we see that $c\geq 1$ and $r_2\geq2c+2>2c+1$. 
    Thus, the aforementioned theorem applies.
    
    Now, let $(r_2, r_3)$ satisfy the second inequality. Observe that in this case $r_2>4$. Hence, using Proposition~\ref{Cr:noDown}, we see that there are no graphs with $r_3 = {r_2 - 1\choose 2}+c$, $c\in \{1, \ldots, 0.5 r_2 - 2\}$. But $4{0.5 r_2\choose 2} - 1 - {r_2 - 1 \choose 2} = 0.5r_2 - 2$, so we get that there are no graphs with ${r_2-1\choose 2}<r_3<4{0.5r_2\choose 2}$.
\end{proof}

\begin{corollary}\label{cor:5.2}
    Let $r_2\ge 3$ be an odd integer. Then there are no graphs with parameters $(r_2, r_3)$, where ${r_2-1\choose 2}<r_3<{r_2\choose 2}$. Moreover, one example of a~graph that corresponds to $r_3 = {r_2\choose 2}$ is $K_{r_2+1}$ and an example that corresponds to $r_3 = {r_2-1\choose 2}$ is $K_{r_2}\square K_2$.
\end{corollary}

\begin{proof}
    It is easy to verify that graphs $K_{r_2+1}$ and $K_{r_2}\square K_2$ are the aforementioned examples.
    
    Using Theorem~\ref{Pr:noDown}, we see that there are no graphs with $r_3={r_2 \choose 2}-c, c\in \{1, \ldots, \frac{r_2 - 1}{2}\}$. In particular, this works for $r_2=3$ (in which case $r_3=2={r_2 \choose 2}-1$). And if $r_2\ge 5$, by virtue of Proposition~\ref{Cr:noDown}, we see that there are no graphs with $r_3={r_2 - 1 \choose 2} + c, c\in \{1, \ldots, \frac{r_2 - 3}{2}\}$. Observing that ${r_2 \choose 2} - 1 - {r_2 - 1 \choose 2} = r_2 - 2 = \frac{r_2 - 1}{2} + \frac{r_2 - 3}{2}$, finishes the proof.
\end{proof}

\section{Tur\'an graphs}\label{sect-Turan}

In this section, we prove two results showing that the graphs $\operatorname{Turan}(2m+2, m+1)$ and $\operatorname{Turan}(3m+3, m+1)$ are uniquely determined by their regularity parameters. We also show that some parameter pairs are not admissible.

\begin{theorem}\label{cor:5.3}
    The graph $\operatorname{Turan}(2m+2, m+1)$ is the unique connected graph with parameters $\left(2m, 4{m \choose 2}\right)$, where $m\geq 2$.
\end{theorem}

\begin{proof}
Let $G$ be a graph with the required parameters. We determine all possible values of $d=K_3\deg(e)$ for $e \in E(G)$ using Lemmas~\ref{K3_edge_upper_bound} and~\ref{K3_edge_lower_bound}
\begin{align*}
    8{m \choose 2} &\leq (2m-1-d)(2m-2) + d(d+1),\\
    d&\leq 2m-1,\\
    d&\in \mathbb{N}\cup\{0\}.
\end{align*}

Thus, $K_3 \deg(e) = d \in \{2m-2, 2m-1\}$. To determine which value actually occurs, we apply Lemma~\ref{K3_deg_eq} and get
\begin{align*}
    8{m \choose 2} &= (2m-2)a+(2m-1)b,\\
    a+b&=2m,
\end{align*}
for some non-negative integer $a$ and $b$.
Observe that $(2m-2)a+(2m-1)b \geq (2m-2)2m = 8{m \choose 2}$, thus $K_3\deg(e)=2m-2$ for all $e \in E(G)$. 

Fix an arbitrary vertex $v \in V(G)$ and consider $G(v)$. We know that $\deg_{G(v)}(u) = K_3\deg(uv)$ for all $u \in N(v)$. Thus $G(v)$ is a $(2m-2)$-regular graph with $2m$~vertices. Hence $G$ must be $\operatorname{Turan}(2m+2, m+1)$.
\end{proof}

Using a similar, but more involved, technique we prove the following result.

\begin{theorem}\label{th:K333}
    Let $r_2\geq7$. If $3 \mid r_2$, then there exists a unique connected graph with parameters $\left(r_2, \frac{r_2(r_2-3)}{2}\right)$, namely $\operatorname{Turan}(r_2+3, \frac{r_2}{3}+1)$.
    If $3 \nmid r_2$, no graph with parameters $\left(r_2, \frac{r_2(r_2-3)}{2}\right)$ exists. 
\end{theorem}

The proof of this theorem proceeds in several steps. We start by determining all possible edge $K_3$-degrees for graphs with these parameters. Next we analyse the local structure of the graph by studying the neighbourhoods $G(v)$ of vertices. This analysis leads to a characterisation of the corresponding Tur\'an graphs.
 
\begin{lemma}\label{lem:typesk3edges}
    Let  $G$ be a graph  with parameters $\left(r_2, \frac{r_2(r_2-3)}{2}\right)$, $r_2\geq 7$. Then for every edge $e \in E(G)$, we have 
    \[K_3\deg(e) \in \{0, r_2-3, r_2-2, r_2-1\}.\]
\end{lemma}

\begin{proof}
  Put $k = K_3\deg(e)$ for some edge $e$. By virtue of Lemmas~\ref{K3_edge_upper_bound}~and~\ref{K3_edge_lower_bound}, we have
\begin{align*}
    0&\leq k \leq r_2-1, \\
    r_2(r_2-3) &\leq (r_2-k-1)(r_2-2)+k(k+1).
\end{align*}
Solving the inequality with respect to $k$, we get $k \in \{0, r_2-3, r_2-2, r_2-1\}$.
\end{proof}

Next we analyse considering the possible structure of the neighbourhood $G(v)$.

\begin{lemma}\label{lem:gen_cl}
    Let $G$ be a graph  with parameters $\left(r_2, \frac{r_2(r_2-3)}{2}\right)$, $r_2\geq 7$. Then for every vertex $v \in V(G)$, the graph $G(v)$ is either $(r_2-3)$-regular or isomorphic to $(K_{r_2-1}-e)\cup K_1$.
\end{lemma}

\begin{proof}
    We fix a vertex $v$ and apply Lemma~\ref{K3_deg_eq} to decompose edges incident to $v$ into four types according to Lemma~\ref{lem:typesk3edges}.  
    Namely, let $a$, $b$, and $c$ denote the numbers of edges incident to $v$ having $K_3$-degrees $r_2-3$, $r_2-2$, $r_2-1$, respectively. This way the other $r_2-a-b-c$ edges incident to $v$ have zero $K_3$-degrees.

    We have the following constraints: 
\begin{align*}
&r_2(r_2-3) = (r_2-3)a+(r_2-2)b+(r_2-1)c,\\
&a+b+c \leq r_2, \\
&a, b, c \in \mathbb{N}\cup\{0\}.
\end{align*}

Assume there are no edges of zero $K_3$-degree, i.e. $a+b+c=r_2$. Then $(r_2-3)a+(r_2-2)b+(r_2-1)c \geq (a+b+c)(r_2-3)=r_2(r_2-3) $, with attaining equality only when $a=r_2$, $b = c = 0$. This is the case when $G(v)$ is $(r_2 - 3)$-regular (note that the $K_3$-degree of an edge $vx$ in $G$ equals the usual degree of a vertex $x$ in $G(v)$).

Assume there are edges of zero $K_3$-degree, i.e. $a+b+c<r_2$. Then the subgraph $G(v)$ has an isolated vertex. Hence, $G(v)$ does not have universal vertices, implying that $c=0$.

If we assume that $a+b\leq r_2-2$, then 
\[(r_2-3)a+(r_2-2)b \leq (r_2-2)^2 < r_2(r_2-3).\] 
This contradiction asserts that $a+b=r_2-1$. 

Hence, we obtain the following system
\begin{equation*}
    \begin{cases}
        r_2-1 = a+b,\\
         r_2(r_2-3) = (r_2-3)a+(r_2-2)b.
    \end{cases}
\end{equation*}
From it we find that $a=2$, $b=r_2-3$.

Thus, summing it up, we see that we have two options so far, either there are no isolated vertices in $G(v)$ and every vertex is of degree $r_2-3$; or there are $2$ vertices of degree $r_2-3$, one isolated vertex, and $r_2-3$ vertices of degree $r_2-2$, which is equivalent of saying that $G(v)\cong (K_{r_2-1}-e)\cup K_1$.
\end{proof}

In the next step we show that only one of the possibilities from the previous lemma can occur.

\begin{lemma}\label{lem:regular}
    Let  $G$ be a graph  with parameters $\left(r_2, \frac{r_2(r_2-3)}{2}\right), r_2\geq 7$. Then for every vertex $v \in V(G)$, the subgraph $G(v)$ is $(r_2-3)$-regular.    
\end{lemma}

\begin{proof}
    Keeping in mind the previous lemma, let us assume that there is $v \in V(G)$ such that $G(v)\cong (K_{r_2-1}-e)\cup K_1$. Without loss of generality, we may assume that $V(G(v)) = \{u, w_1, w_2, g_1, \ldots, g_{r_2-3}\}$, where $u$ is the isolated vertex, $g_1, \ldots, g_{r_2-3}$ are vertices from $K_{r_2-1}-e$, and $w_1$ and $w_2$ are also from $K_{r_2-1} -e$ but they are not adjacent.

    Now, we should change our perspective to $G(w_1)$. We know that $$V(G(w_1))\cap V(G[v]) = \{v, g_1, \ldots, g_{r_2-3}\},$$
    as well as that $$\deg_{G(w_1)} (g_j) = |N(g_j)\cap V(G(w_1))| \geq |\{v, g_1, \ldots, g_{r_2-3}\}| - 1 = r_2-3.$$
    Which means that we need to introduce two new vertices, let us call them $a \in G(w_1)$ and $b \in G(w_1)$. 

    Now there are two possibilities. The first one, is that neither $a$ nor $b$ are adjacent to any of $g_j$, it would mean that $\deg_{G(w_1)}(a)$ and $\deg_{G(w_1)}(b)$ are simultaneously either $0$ or $1$, but according to Lemma~\ref{lem:gen_cl} neither options are possible. 
    
    The second option is that either $a$ or $b$ (or both) are adjacent to some $g_k$. But it would mean that $\deg_{G(w_1)} (g_k) > r_2-3$. Thus, $G(w_1)\cong (K_{r_2-1}-e)\cup K_1$ by Lemma~\ref{lem:gen_cl}. Without loss of generality, we shall assume that $b$ is isolated in $G(w_1)$, and therefore, $a$ is adjacent to every vertex $g_j$.

    Finally, let us shift our attention once again to  $G(g_1)$, we can see that $$\{a, v, w_1, w_2, g_2, \ldots, g_{r_2-3}\} =  V(G(g_1))$$ and that $\deg_{G(g_1)} (g_2) = r_2-1$, but it is not possible according to Lemma~\ref{lem:gen_cl}. 
\end{proof}

Having established the regularity of $G(v)$, we now determine its precise structure.

\begin{lemma}\label{lemma:finallocat}
    Let $G$ be a graph with parameters $\left(r_2, \frac{r_2(r_2-3)}{2}\right), r_2\geq 7$. Then for every vertex $v \in V(G)$, the subgraph $G(v)$ is the complement of a disjoint union of triangles. 
\end{lemma}

\begin{proof}
By Lemma~\ref{lem:regular}, the graph $G(v)$ is $(r_2-3)$-regular. Hence, its complement is $2$-regular and therefore is a disjoint union of cycles. Hence, $\overline {G(v)} = C_{n_1}\cup\ldots \cup C_{n_k},$ where $n_i\geq 3$ and $n_1+\ldots+n_k = r_2$. 

Thus, we need to prove that there cannot be any index $i$ with $n_i>3$. By way of contradiction, without loss of generality, assume that $n_1>3$. 

Let us denote the vertices in $G(v)$ as follows \[V(G(v))=\{g_1^1, \ldots, g_{n_1}^1; \ldots; g_{1}^k, \ldots, g_{n_k}^k\},\]
in a way so that $g_i^s$, $1\leq i\leq n_s$ constitute a cycle $C_{n_s}$ in $\overline{G(v)}$ for a fixed $1\leq s\leq k$.

As before, let us shift our perspective to $G(g_2^1)$, we can see that 
$$V(G(g_2^1))\cap V(G[v]) = \{v, g_4^1, \ldots, g_{n_1}^1, \ldots, g_{n_k}^k\}.$$
Which means that we need to introduce exactly two new vertices $a$ and $b$. We note that $C_{n_2}, \ldots, C_{n_k}$ remain cycles in $\overline{G(g_2^1)}$, as such $a$ and $b$ must be adjacent to every $g_i^s, s>1$. We also note that $\{av, bv\} \cap E(G) = \emptyset$. By way of contradiction, let us assume that $a$ and $b$ are not adjacent in $G$, it would mean that $a, b, v$ constitute a cycle in $\overline{G(g_2^1)}$, which would mean that 
$$\deg_{G(g_2^1)} (g_4^1) \! \geq \!
\begin{cases}
    |\{a, b, v, g_6^1, \ldots, g_{n_1}^1, \ldots, g_{n_k}^k\}|  = 3+r_2-5=r_2-2, ~ n_1 \!\geq\! 6,\\
    |\{a, b, v, g_1^2, \ldots, g_{n_2}^2, \ldots, g_{n_k}^k\}| \geq 3+r_2-5=r_2-2, ~ n_1\!\leq\! 5,\\
\end{cases}$$
thus $G(g_2^1)$ would not be regular. Thus, $a$ and $b$ are adjacent in $G$, which leads to the conclusion, without loss of generality, that $a$ and $g_4^1$ are adjacent in $\overline{G(g_2^1)}$, and $b$ and $g_{n_1}^1$ are adjacent in $\overline{G(g_2^1)}$ (note $n_1$ might be $4$).

We need to consider two separate cases based on if there are more than one cycle in $G(v)$.

Assume that there are at least two cycles in $\overline{G(v)}$, then we can consider $G(g_1^2)$. We know that $$\deg_{G(g_1^2)} (g_4^1) \begin{cases}
    \leq r_2-1 - |\{a, b, g_1^1, g_3^1\}| = r_2-5, n_1=4,\\
    \leq r_2-1 - |\{a, g_5^1, g_3^1\}| = r_2-4, n_1\geq 5,
\end{cases}$$
thus $\deg_{G(g_1^2)} (g_4^1) < r_2-3$, which fails to make $G(g_1^2)$ $(r_2-3)$-regular.

Assume that there is only one cycle in $\overline{G(v)}$, it means that $G(v) = \overline{C_{n_1}}$ with $n_1\geq 7$, it means that we can consider $G(g_7^1)$. We see that $$\deg_{G(g_7^1)} (g_4^1) \leq r_2-1 - |\{a, g_3^1, g_5^1\}| = r_2-4,$$
which means that $G(g_7^1)$ cannot be $(r_2-3)$-regular.
\end{proof}

To complete the proof of Theorem~\ref{th:K333}, we use a characterisation of Tur\'an graphs determined by their local structure.

\begin{proposition}[{\cite[Proposition 1.1.5]{BCN:89}}]\label{pr:forfinal}
    Let $G$ be a connected graph that is locally complete multipartite. Then $G$ is either triangle-free or complete multipartite. In particular, if $G$ is locally $K_{m_1, \ldots, m_k}$ then all $m_i$ are equal (to $m$, say), and $G \cong K_{(k+1)\times m}$.
\end{proposition}

We now formulate the final step needed to complete the proof.

\begin{lemma}\label{lemma:turan}
Let $G$ be a connected graph such that the neighbourhood of every vertex of $G$ is isomorphic to $\overline{k C_3}$, for some $k \in \mathbb{N}$. Then $G \cong \operatorname{Turan}(3k+3, k+1)$.
\end{lemma} 
\begin{proof}
    Using Proposition~\ref{pr:forfinal}, we readily see that $G\cong K_{(k+1)\times 3}$ which is $\operatorname{Turan}(3k+3, k+1)$.
\end{proof}

Combining Lemmas~\ref{lemma:finallocat} and~\ref{lemma:turan} completes the proof of Theorem~\ref{th:K333}. This immediately implies the following corollary.

Combining Lemmas~\ref{lemma:finallocat} and \ref{lemma:turan} directly concludes the proof of Theorem~\ref{th:K333}, which in turn implies the following corollary.

\begin{corollary}\label{cor:main}
   $\operatorname{Turan}(3(m+1), m+1)$, $m\geq 3$ is the unique connected graph with parameters $\left(m, \frac{m(m-3)}{2}\right)$.
\end{corollary}

Finally, it is worth mentioning an interesting conjecture from~\cite{SheffieldXi:25} stating whenever a graph with parameters $(r_2, r_3)$ exists, there also exists an abelian Cayley graph with the same parameters $(r_2, r_3)$. The results obtained in this section are consistent with this conjecture, since the Tur\'an graphs considered above are abelian Cayley graphs.

\section{Open questions}\label{sect-4}

In this section, we present several open questions about regular $K_3$-regular graphs for further research.

Note that the bounds from Corollaries~\ref{cor:5.1} and~\ref{cor:5.2} are insufficient for the complete characterisation of the parameters for the existence of regular $K_3$-regular graphs. For example, the case $r_2 = 7$, $r_3 = 14$ is not covered by Corollaries~\ref{cor:5.1} and~\ref{cor:5.2}, yet there are no such graphs (by Theorem~\ref{th:K333}). The cases $(8,17)$, $(9,17)$, $(9,23)$, $(10,23)$ are the smallest pairs for which we do not know whether such graphs exist. This begs the first two questions.
	
\textbf{Question 1.} Provide a complete list of the values of parameters for which there exist regular $K_3$-regular graphs.

\textbf{Question 2.} For an admissible pair $(r_2,r_3)$, characterize all connected regular $K_3$-regular graphs with these parameters.

Now recall that Theorems~\ref{cor:5.3} and~\ref{cor:main} give a criterion for $\operatorname{Turan}(2m+2, m+1)$ and $\operatorname{Turan}(3m+3, m+1).$
            
\textbf{Question 3.} Can one characterize graphs $\operatorname{Turan}(n, r)$ in a similar manner as it presented in Corollary~\ref{cor:5.3} for particular Tur\'an graphs?

\section*{Acknowledgments} 
	The authors are deeply grateful to the Ukrainian Armed Forces for keeping Leliukhivka, Kyiv, Kharkiv, and Khmelnytskyi safe, which gave us the opportunity to work on this paper. We are grateful to Helmut Ruhland for inspiring discussions that motivated this work. We also thank Vyacheslav Boyko for helpful comments that improved the presentation, and Andrii Serdiuk for a~careful reading and insightful remarks on the mathematical content. Artem Hak has been supported by EDUFI (TFK programme, 12/221/2023).

\newpage

\appendix
\refstepcounter{section}
\section*{Appendix~\thesection}\label{appendix}

\newcommand{\Expand}{\cellcolor{green!25}\ensuremath{\rightarrow}}

\begin{table}[!ht]
	\caption{Admissible parameters for regular $K_3$-regular graphs. The table is split into two parts for readability. The left part of the second block is omitted since all corresponding entries are $\mathrm{No}(1)$.}
	\label{tab:admissible}
	\renewcommand{\arraystretch}{1.0}
	\begin{tabular}{|c|c|c|c|c|c|c|c|c|c|}
		\hline
		$r_3 \setminus r_2$ & $2$ & $3$ & $4$ & $5$ & $6$ & $7$ & $8$ \\ \hline
		$1$ & \cellcolor{green!25}$K_3$ & \Expand & \Expand & \Expand & \Expand & \Expand & \Expand \\ \hline
		$2$ & \No{1} & \No{3} & \cellcolor{green!25}$K_3 \square K_3$ & \Expand & \Expand & \Expand & \Expand \\ \hline
		$3$ & \No{1} & \cellcolor{green!25}$K_4$ & \Expand & \Expand & \Expand & \Expand & \Expand \\ \hline
		$4$ & \No{1} & \No{1} & \cellcolor{green!25}$\operatorname{Turan}(6, 3)$ & \Expand & \Expand & \Expand & \Expand \\ \hline
		$5$ & \No{1} & \No{1} & \No{2} & \cellcolor{green!25}$G_1$ & \cellcolor{green!25}$\operatorname{Turan}(6, 3) \square K_3$ & \Expand & \Expand \\ \hline
		$6$ & \No{1} & \No{1} & \cellcolor{green!25}$K_5$ & \Expand & \Expand & \Expand & \Expand \\ \hline
		$7$ & \No{1} & \No{1} & \No{1} & \No{3} & \cellcolor{green!25}$K_5 \square K_3$ & \Expand & \Expand \\ \hline
		$8$ & \No{1} & \No{1} & \No{1} & \No{3} & \cellcolor{green!25}$G_2$ & \Expand & \cellcolor{green!25}$K_5 \square K_3 \square K_3$ \\ \hline
		$9$ & \No{1} & \No{1} & \No{1} & \No{3} & \cellcolor{green!25}$G_3$ & \cellcolor{green!25}$K_5 \square  K_4$ & \Expand \\ \hline
		$10$ & \No{1} & \No{1} & \No{1} & \cellcolor{green!25}$K_6$ & \Expand & \Expand & \Expand \\ \hline
		$11$ & \No{1} & \No{1} & \No{1} & \No{1} & \No{2} & \cellcolor{green!25}$K_6 \square K_3$ & \Expand \\ \hline
		$12$ & \No{1} & \No{1} & \No{1} & \No{1} & \cellcolor{green!25}$\operatorname{Turan}(8,4) $& \Expand & \Expand \\ \hline
		$13$ & \No{1} & \No{1} & \No{1} & \No{1} & \No{2} & \cellcolor{green!25}$G_4$ & \cellcolor{green!25}$K_6 \square K_4$ \\ \hline
		$14$ & \No{1} & \No{1} & \No{1} & \No{1} & \No{2} & \No{4} & \cellcolor{green!25}$G_5$ \\ \hline
		$15$ & \No{1} & \No{1} & \No{1} & \No{1} & \cellcolor{green!25}$K_7$ & \Expand & \Expand \\ \hline
	\end{tabular}

	\vspace{1.0cm}
	
	\setlength{\tabcolsep}{1.6pt}
	\begin{tabular}{|c|c|c|c|c|c|c|c|c|c|}
		\hline
		$r_3 \setminus r_2$ & 6 & 7 & 8 & 9 & 10 & 11 & 12 \\ \hline
		$15$ & \cellcolor{green!25}$K_7$ & \Expand & \Expand & \Expand & \Expand & \Expand & \Expand \\ \hline
		$16$ & \No{1} & \No{3} & \cellcolor{green!25}$\operatorname{Turan}(12,3) $ & \Expand & \Expand & \Expand & \Expand \\ \hline
		$17$ & \No{1} & \No{3} & \textit{Unknown} & \textit{Unknown} & \cellcolor{green!25}$\operatorname{Turan}(12,3) \square K_3$ & \Expand & \Expand \\ \hline
		$18$ & \No{1} & \No{3} & \cellcolor{green!25}$C^{1, 2, 3, 4}_{13}$ & \Expand & \Expand & \Expand & \Expand \\ \hline
		$19$ & \No{1} & \No{3} & \cellcolor{green!25}$C^{1, 2, 3, 4}_{12}$ & \Expand & \Expand & \Expand & \Expand \\ \hline
		$20$ & \No{1} & \No{3} & \No{4} & \cellcolor{green!25}$G_6$ & \Expand & \Expand & \Expand \\ \hline
		$21$ & \No{1} & \cellcolor{green!25}$K_8$ & \Expand & \Expand & \Expand & \Expand & \Expand \\ \hline
		$22$ & \No{1} & \No{1} & \No{2} & \Expand & \Expand & \Expand & \Expand \\ \hline
		$23$ & \No{1} & \No{1} & \No{2} & \textit{Unknown} & \textit{Unknown} & \Expand & \Expand \\ \hline
		$24$ & \No{1} & \No{1} & \cellcolor{green!25}$\operatorname{Turan}(10,5)$ & \Expand & \Expand & \Expand & \Expand \\ \hline
		$25$ & \No{1} & \No{1} & \No{2} & \textit{Unknown} & \Expand & \Expand & \Expand \\ \hline
		$26$ & \No{1} & \No{1} & \No{2} & \textit{Unknown} & \textit{Unknown} & \textit{Unknown} & \cellcolor{green!25}$\operatorname{Turan}(15,3) \square K_3$ \\ \hline
		$27$ & \No{1} & \No{1} & \No{2} & \cellcolor{green!25} $\operatorname{Turan}(12,4)$ & \Expand & \Expand & \Expand \\ \hline
		$28$ & \No{1} & \No{1} & \cellcolor{green!25}$K_9$ & \Expand & \Expand & \Expand & \Expand \\ \hline
	\end{tabular}
\end{table}

\medskip
\noindent
\textbf{Legend for Table~\ref{tab:admissible}.}
\begin{itemize}
	\item \No{1}: Lemma~\ref{K3_upper_bound};
	\item \No{2}: Corollary~\ref{cor:5.1};
	\item \No{3}: Corollary~\ref{cor:5.2};
	\item \No{4}: Theorem~\ref{th:K333};
	\item $C_n^{s_1, \ldots, s_k}$ denotes a circulant graph with jumps $s_1, \ldots, s_k$ on $n$ vertices;
	\item $\rightarrow$: blow-up construction, i.e., $G \square K_2$ applied to the graph immediately to the left.
\end{itemize}


The graph $G_1$ is an icosahedron, and $G_2$, $G_3$, $G_4$, $G_5$ and $G_6$ are depicted in Figure~\ref{fig:2x2}. Moreover, one can find adjacency data for these graphs in \href{https://github.com/DenisLohvynov/Regular-K3-regular-graphs-examples}{this GitHub repository}.

We use the heuristic search framework introduced in~\cite{Hak:25}. In the present work, we rely on the same code. The search starts from an arbitrary $r_2$-regular graph and applies a standard $2$-switch mutation, which preserves regularity.

To guide the search towards $K_3$-regularity, we use the following fitness function. Let $G$ be an $r_2$-regular graph on $n$ vertices. For a prescribed value $r_3$, we define
\[
\operatorname{fitness}(G) = \frac{1}{n}
\sum_{v \in V(G)} \frac{1}{ \left( \lvert K_3 \deg(v) - r_3 \rvert + 1 \right)^2 }.
\]

Thus, graphs whose vertex $K_3$-degrees are closer to $r_3$ receive higher fitness values; the maximum is attained precisely when $G$ is $K_3$-regular.

\begin{figure}[htbp]
	\centering
	\begin{subfigure}{0.45\textwidth}
		\centering
		\begin{tikzpicture}[
			scale=3,
			every node/.style={circle, draw, fill=black, inner sep=1.8pt},
			every edge/.style={draw, thick}
			]
			
			\foreach \i in {0,...,17} {
				\node (v\i) at ({360*\i/18}:1) {};
			}
			
			\foreach \u/\v in {
				0/1,0/2,0/3,0/13,0/14,0/16,
				1/2,1/3,1/14,1/15,1/16,
				2/3,2/4,2/5,2/7,
				3/4,3/5,3/6,
				4/5,4/6,4/7,4/11,
				5/6,5/7,5/14,
				6/7,6/8,6/9,
				7/8,7/9,
				8/9,8/10,8/11,8/17,
				9/11,9/12,9/17,
				10/11,10/12,10/13,10/15,10/17,
				11/12,11/17,
				12/13,12/15,12/17,
				13/14,13/15,13/16,
				14/15,14/16,
				15/16,
				16/17%
			}{
				\draw (v\u) -- (v\v);
			}
			
		\end{tikzpicture}
		\caption{The graph $G_2$ with parameters $(6,8)$.}
		\label{fig:G2}
		
	\end{subfigure}
	\hfill
	\begin{subfigure}{0.45\textwidth}
		\centering
		\begin{tikzpicture}[
			scale=3,
			every node/.style={circle, draw, fill=black, inner sep=1.8pt},
			every edge/.style={draw, thick}
			]
			
			\foreach \i in {0,...,9} {
				\node (v\i) at ({360*\i/10}:1) {};
			}
			
			\foreach \u/\v in {
				0/1,0/2,0/3,0/4,0/7,0/8,
				1/3,1/4,1/6,1/7,1/8,
				2/3,2/4,2/5,2/7,2/9,
				3/4,3/5,3/9,
				4/6,4/9,
				5/6,5/7,5/8,5/9,
				6/7,6/8,6/9,
				7/8,
				8/9%
			}{
				\draw (v\u) -- (v\v);
			}
			
		\end{tikzpicture}
		\caption{The graph $G_3$ with parameters $(6,9)$.}
		\label{fig:G3}
	\end{subfigure}
	
	\medskip
	
	\begin{subfigure}{0.45\textwidth}
		\centering
		\begin{tikzpicture}[
			scale=3,
			every node/.style={circle, draw, fill=black, inner sep=1.8pt},
			every edge/.style={draw, thick}
			]
			
			\foreach \i in {0,...,11} {
				\node (v\i) at ({360*\i/12}:1) {};
			}
			
			\foreach \u/\v in {
				0/1,0/2,0/4,0/5,0/6,0/7,0/11,%
				1/2,1/3,1/5,1/6,1/7,1/8,%
				2/3,2/4,2/7,2/8,2/9,%
				3/4,3/5,3/8,3/9,3/10,%
				4/5,4/9,4/10,4/11,%
				5/6,5/10,5/11,%
				6/7,6/8,6/10,6/11,%
				7/8,7/9,7/11,%
				8/9,8/10,%
				9/10,9/11,%
				10/11%
			}{
				\draw (v\u) -- (v\v);
			}
			
		\end{tikzpicture}
		\caption{The graph $G_4$ with parameters $(7,13)$.}
		\label{fig:G4}
	\end{subfigure}
	\hfill
	\begin{subfigure}{0.45\textwidth}
		\centering
		\begin{tikzpicture}[
			scale=3,
			every node/.style={circle, draw, fill=black, inner sep=1.8pt},
			every edge/.style={draw, thick}
			]
			
			\foreach \i in {0,...,14} {
				\node (v\i) at ({360*\i/15}:1) {};
			}
			
			\foreach \u/\v in {
				0/1,0/2,0/4,0/5,0/6,0/7,0/12,0/14,
				1/2,1/5,1/8,1/9,1/12,1/13,1/14,
				2/5,2/6,2/10,2/11,2/12,2/13,
				3/4,3/7,3/9,3/10,3/11,3/12,3/13,3/14,
				4/5,4/6,4/8,4/10,4/12,4/14,
				5/8,5/9,5/10,5/11,
				6/7,6/9,6/11,6/13,6/14,
				7/9,7/10,7/11,7/12,7/13,
				8/9,8/10,8/11,8/12,8/14,
				9/11,9/14,
				10/12,10/13,
				11/13,
				13/14%
			}{
				\draw (v\u) -- (v\v);
			}
			
		\end{tikzpicture}
		\caption{The graph $G_5$ with parameters $(8,14)$.}
		\label{fig:G5}
		
	\end{subfigure}
	\medskip

	\begin{subfigure}{1\textwidth}
		\centering
		\begin{tikzpicture}[
			scale=1,
			every node/.style={circle, draw, fill=black, inner sep=1.8pt},
			every edge/.style={draw, thick}
			]
			
			\foreach \i in {0,...,17} {
				\node (v\i) at ({360*\i/18}:3.5) {};
			}
			
			\foreach \u/\v in {
				0/3,0/7,0/8,0/10,0/11,0/13,0/14,0/16,0/17,1/4,1/6,1/8,1/9,1/11,1/12,1/14,1/15,1/17,2/5,2/6,2/7,2/9,2/10,2/12,2/13,2/15,2/16,3/7,3/8,3/10,3/11,3/13,3/14,3/16,3/17,4/6,4/8,4/9,4/11,4/12,4/14,4/15,4/17,5/6,5/7,5/9,5/10,5/12,5/13,5/15,5/16,6/9,6/13,6/14,6/16,6/17,7/10,7/12,7/14,7/15,7/17,8/11,8/12,8/13,8/15,8/16,9/13,9/14,9/16,9/17,10/12,10/14,10/15,10/17,11/12,11/13,11/15,11/16,12/15,13/16,14/17%
			}{
				\draw (v\u) -- (v\v);
			}
		\end{tikzpicture}
		
		\caption{The graph $G_6$ with parameters $(9,20)$.}
		\label{fig:G6}
		
	\end{subfigure}
	
	\caption{The graphs $G_2$, $G_3$, $G_4$ and $G_5$ from Table~\ref{tab:admissible}.}
	\label{fig:2x2}
\end{figure}
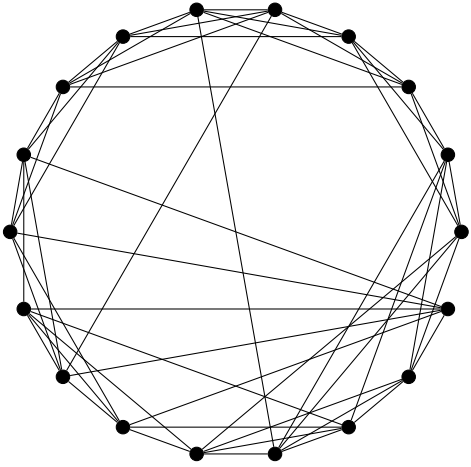
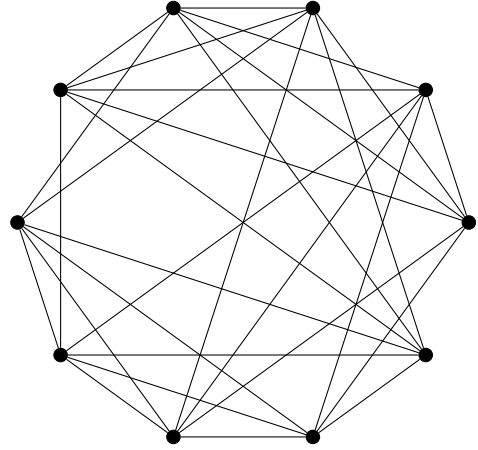
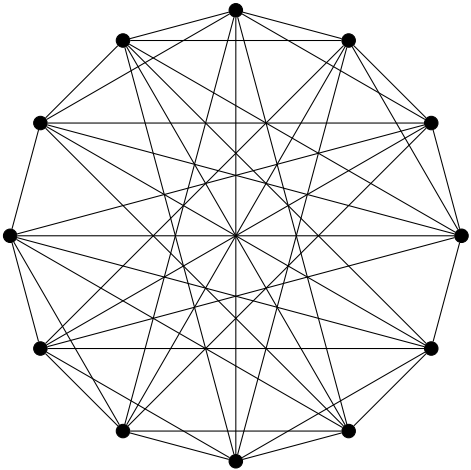
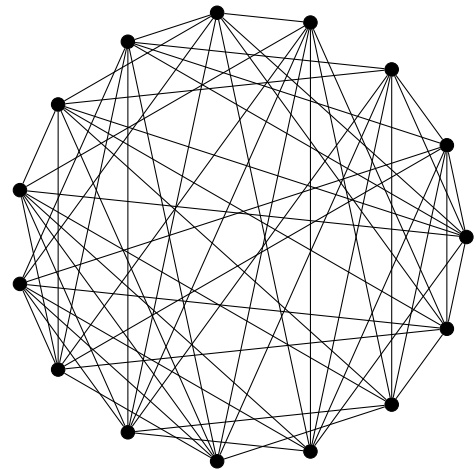

\end{document}